\documentclass[reqno]{amsart}
\usepackage{amssymb,setspace}
\usepackage{ifpdf}
\ifpdf
 \usepackage[hyperindex,pagebackref]{hyperref}
\else
 \expandafter\ifx\csname dvipdfm\endcsname\relax
 \usepackage[hypertex,hyperindex,pagebackref]{hyperref}
 \else
 \usepackage[dvipdfm,hyperindex,pagebackref]{hyperref}
 \fi
\fi
\allowdisplaybreaks[4]
\numberwithin{equation}{section}
\theoremstyle{plain}
\newtheorem{thm}{Theorem}[section]

\theoremstyle{remark}
\newtheorem{rem}{Remark}[section]
\DeclareMathOperator{\td}{d}

\begin{document}

\title[Integral representations and complete monotonicity]
{Integral representations and complete monotonicity related to the remainder of Burnside's formula for the gamma function}

\author[F. Qi]{Feng Qi}
\address{College of Mathematics, Inner Mongolia University for Nationalities, Tongliao City, Inner Mongolia Autonomous Region, 028043, China}
\email{\href{mailto: F. Qi <qifeng618@gmail.com>}{qifeng618@gmail.com}, \href{mailto: F. Qi <qifeng618@hotmail.com>}{qifeng618@hotmail.com}, \href{mailto: F. Qi <qifeng618@qq.com>}{qifeng618@qq.com}}
\urladdr{\url{http://qifeng618.wordpress.com}}

\begin{abstract}
In the paper, the authors establish integral representations of some functions related to the remainder of Burnside's formula for the gamma function and find the (logarithmically) complete monotonicity of these and related functions. These results extend and generalize some known conclusions.
\end{abstract}



\thanks{This paper was typeset using \AmS-\LaTeX}

\maketitle

\section{Motivation and main results}

A function $f$ is said to be completely monotonic on an interval $I$ if $f$ has derivatives of all orders on $I$ and $(-1)^{n}f^{(n)}(x)\ge0$ for all $x \in I$ and $n \ge0$. A function $f$ is said to be absolutely monotonic on an interval $I$ if $f$ has derivatives of all orders on $I$ and $f^{(n)}(x)\ge0$ for all $x \in I$ and $n \ge0$. A positive function $f(x)$ is said to be logarithmically completely monotonic on an interval $I\subseteq\mathbb{R}$ if it has derivatives of all orders on $I$ and its logarithm $\ln f(x)$ satisfies $(-1)^k[\ln f(x)]^{(k)}\ge0$ for all $k\in\mathbb{N}$ on $I$. For more information on these kinds of functions, please refer to the papers and monographs~\cite{absolute-mon-simp.tex, mpf-1993, BAustMS-5984-RV.tex, gamma-sqrt-ratio.tex, SCM-2012-0142.tex, Schilling-Song-Vondracek-2nd, widder} and plenty of references cited therein.
\par
It is well known that the classical Euler's gamma function may be defined by
\begin{equation}\label{egamma}
\Gamma(x)=\int^\infty_0t^{x-1} e^{-t}\textup{\,d}t
\end{equation}
for $x>0$. The logarithmic derivative of $\Gamma(x)$, denoted by $\psi(x)=\frac{\Gamma'(x)}{\Gamma(x)}$, is called the psi or digamma function, and $\psi^{(k)}(x)$ for $k\in \mathbb{N}$ are called the polygamma functions. The noted Binet's formula~\cite[p.~11]{magnus} states that
\begin{equation}\label{ebinet}
\ln \Gamma(x)= \biggl(x-\frac12\biggr)\ln x-x+\ln \sqrt{2\pi}\,+\theta(x)
\end{equation}
for $x>0$, where $\Gamma(x)=\int^\infty_0t^{x-1} e^{-t}\td t$ stands for Euler's gamma function and
\begin{equation}\label{ebinet1}
\theta(x)=\int_{0}^{\infty}\left(\frac{1}{e^{t}-1}-\frac{1}{t}
+\frac12\right)\frac{e^{-xt}}{t}\td t
\end{equation}
is called the remainder of Binet's formula~\eqref{ebinet}.
By the way, some functions related to the function $\frac{1}{e^{t}-1}-\frac{1}{t}+\frac12$ in the formula~\eqref{ebinet1} have been investigated, applied, and surveyed in~\cite{emv-log-convex-simple.tex, Bernoulli-Stirling2-3P.tex, exp-derivative-sum-Combined.tex, TamsuiOxf.J.Math.Sci.(Mar13-07).tex, Bernoulli-No-Int-New.tex, Eight-Identy-More.tex, 1st-Sirling-Number-2012.tex, Qi-Springer-2012-Srivastava.tex} and many references listed therein.
\par
For real numbers $p>0$, $q\in\mathbb{R}$, and $r\ne0$, define
\begin{equation}
f_{p,q,r}(x)=r[\theta(px)-q\theta(x)]
\end{equation}
on $(0,\infty)$. In~\cite{binet, TamsuiOxf.J.Math.Sci.(Mar13-07).tex} and~\cite[Section~5]{Chen-Qi-Srivastava-09.tex}, the complete monotonicity of $f_{p,q,r}(x)$ and the star-shaped and subadditive properties of $\theta(x)$ were established.
\par
In~\cite[p.~892]{gr} and~\cite[p.~17]{magnus}, it is given that
\begin{equation}
\psi(x)=\ln x-\frac1{2x}-2\int_{0}^{\infty}\frac{t\td t} {(t^2 +x^2)(e^{2\pi t}-1)}
\end{equation}
and
\begin{equation}
\psi\biggl(x+\frac12\biggr)=\ln
x+2\int_{0}^{\infty}\frac{t\td t} {(t^2 +4x^2)(e^{\pi t}+1)} \label{e2d}
\end{equation}
for $x>0$.
For $p>0$ and $q\in\mathbb{R}$, let
\begin{equation}
\Lambda_{p,q}(x)=\lambda(px)-q\lambda(x)\quad\text{and}\quad
\Phi_{p,q}(x)=\phi(px)-q\phi(x)
\end{equation}
on $(0,\infty)$, where
\begin{gather}
\lambda(x)=\int_{0}^{\infty}\frac{t\td t} {(t^2 +x^2)(e^{2\pi t}-1)}\quad
\text{and}\quad \phi(x)=\int_{0}^{\infty}\frac{t\td t} {(t^2 +4x^2)(e^{\pi t}+1)}.
\end{gather}
In~\cite{psi-reminders.tex-ijpams, psi-reminders.tex-rgmia}, it was obtained that
\begin{enumerate}
\item
the function $\Lambda_{p,q}(x)$ is positive and decreasing in $x\in(0,\infty)$
if
\begin{enumerate}
\item
either $q\le0$,
\item
or $0<p<1$ and $pq\le1$,
\item
or $0<q=\frac1{p^2}\le1$;
\end{enumerate}
\item
the function $\Lambda_{p,q}(x)$ is negative and increasing in $x\in(0,\infty)$ if
\begin{enumerate}
\item
either $p\ge1$ and $pq\ge1$,
\item
or $\frac1{p^2}=q\ge1$;
\end{enumerate}
\item
the function $\Phi_{p,q}(x)$ is positive and decreasing in $x\in(0,\infty)$ if
\begin{enumerate}
\item
either $p\ge1$ and $q\le0$,
\item
or $0<p<1$ and $q\le1$,
\item
or $p^2q<1$ and $q(p^2-1)[(1+3q)p^2-4]\le0$,
\item
or $p^2q=1$ and $0<q\le1$;
\end{enumerate}
\item
the function $\Phi_{p,q}(x)$ is negative and increasing in $x\in(0,\infty)$ if
\begin{enumerate}
\item
either $4\le p^2(1+3q)\le1+3q$,
\item
or $p>1$ and $q\ge1$.
\end{enumerate}
\end{enumerate}
\par
In~\cite{Extended-Binet-remiander-comp.tex-Math-Slovaca} and its preprint~\cite{Extended-Binet-remiander-comp.tex-arXiv}, some more properties on the remainder of Binet's formula~\eqref{ebinet} were further obtained.
\par
Binet's formula~\eqref{ebinet} may be reformulated as
\begin{equation}\label{theta(x)-gamma}
\Gamma(x+1)=\sqrt{2\pi x}\,\Bigl(\frac{x}e\Bigr)^{x}e^{\theta(x)}.
\end{equation}
We also call $\theta(x)$ the remainder of Stirling's formula
\begin{equation}
n!\sim\sqrt{2\pi n}\,\Bigl(\frac{n}e\Bigr)^{n}, \quad n\to\infty
\end{equation}
established by James Stirling in~1764.
When replacing $\theta(x)$ by $\frac{1}{12}\psi'(x+\alpha)$, it was proved in~\cite{AMIS042013A.tex, Merkle-Convexity2Complete-Mon.tex, Sevli-Batir-MCM-2011} that the function
\begin{equation}
F_{\alpha}(x) =\frac{e^x\Gamma(x+1)}{x^x\sqrt{2\pi x}\,e^{\psi'(x+\alpha)/12}}
\end{equation}
is logarithmically completely monotonic on $(0,\infty)$ if and only if $\alpha\ge\frac12$ and that the function $\frac1{F_{\alpha}(x)}$ is logarithmically completely monotonic on $(0,\infty)$ if and only if $\alpha=0$. Consequently, the double inequality
\begin{equation}\label{double-ineqSevli-BatirModel}
\exp \biggl(\frac{1}{12}\psi'\biggl(x+\frac12\biggr) \biggr) <\frac{e^x\Gamma(x+1)}{x^x\sqrt{2\pi x}\,} <\exp \biggl(\frac{1}{12}\psi'(x) \biggr)
\end{equation}
was derived in~\cite[Corollary~2.1]{Sevli-Batir-MCM-2011}.
\par
In~\cite{Mortici-ps-Qi.tex}, the following conclusions were obtained.
\begin{enumerate}
\item
As $n\to \infty$, the asymptotic formula
\begin{equation} \label{a}
n! \sim \frac{n^n}{e^n}\sqrt{2\pi n}\,\exp \biggl(\frac{1}{12}\psi'\biggl(n+\frac12\biggr) \biggr)
\end{equation}
is the most accurate one among all approximations of the form
\begin{equation}\label{aa}
n! \sim \frac{n^n}{e^n}\sqrt{2\pi n}\,\exp \biggl(\frac{1}{12}\psi'(n+a)\biggr),
\end{equation}
where $a\in\mathbb{R}$;
\item
As $x\to\infty$, we have
\begin{multline}
\Gamma(x+1)\sim\sqrt{2\pi}\,x^{x+1/2}\exp \biggl(\frac{1}{12}\psi'\biggl(x+\frac12\biggr)-x+\frac{1}{240}\frac1{x^{3}}\\*
-\frac{11}{6720}\frac1{x^{5}}+\frac{107}{80640}\frac1{x^{7}} -\frac{2911}{1520640}\frac1{x^{9}}+\dotsm\biggr);
\end{multline}
\item
For every integer $x\ge 1$, we have
\begin{equation}\label{ii-continuous}
\exp \biggl(\frac{1}{240x^{3}}-\frac{11}{6720x^{5}}\biggr)
<\frac{e^{x}\Gamma(x+1)}{x^{x}\sqrt{2\pi x}\exp \bigl( \frac{1}{12}\psi'\bigl(x+\frac12\bigr) \bigr)}
<\exp \biggl(\frac{1}{240x^{3}}\biggr).
\end{equation}
\end{enumerate}
\par
For $\alpha\in\mathbb{R}$, let
\begin{equation}
g_{\alpha}(x)=\frac{e^x\Gamma(x+1)}{(x+\alpha)^{x+\alpha}}
\end{equation}
on the interval $(\max\{0,-\alpha\},\infty)$. In \cite{s-guo-ijpam, Guo-Qi-Srivastava2007.tex}, it was showed that the function $g_{\alpha}(x)$ is logarithmically completely monotonic if and only if $\alpha\ge1$ and that the function $\frac1{g_{\alpha}(x)}$ is logarithmically completely monotonic if and only if $\alpha\le\frac12$.
\par
In~\cite[Theorem~1]{Bukac-Buric-Elezovic-MIA-2011}, the double inequality
\begin{multline}\label{p.236-Theorem-1}
 \biggl(x-\frac12\biggr)\biggl[\ln\biggl(x-\frac12\biggr)-1\biggr]+\ln\sqrt{2\pi}\, -\frac1{24(x-1)}\le\ln\Gamma(x)\\*
 \le\biggl(x-\frac12\biggr)\biggl[\ln\biggl(x-\frac12\biggr)-1\biggr] +\ln\sqrt{2\pi}\,-\frac1{24\bigl(\sqrt{x^2+x+1/2}\,-1/2\bigr)}, \quad x>1
\end{multline}
was obtained, which may be rewritten as
\begin{equation}\label{p.236-Theorem-1-rew}
e^{-1/24x}< \frac{e^{x+1/2}\Gamma(x+1)}{\sqrt{2\pi}\,(x+1/2)^{x+1/2}} \le e^{-1/24\bigl(\sqrt{x^2+3x+5/2}\,-1/2\bigr)},\quad x>0.
\end{equation}
\par
In~\cite[p.~1774, Theorem~2.3]{Sevli-Batir-MCM-2011}, the function
\begin{equation}
 H(x)=\frac{e^{x+1/24(x+1/2)}\Gamma(x+1)}{(x+1/2)^{x+1/2}}
\end{equation}
was proved to be logarithmically completely monotonic on $(0,\infty)$. Consequently, it was deduced in~\cite[Corollary~2.4]{Sevli-Batir-MCM-2011} that the double inequality
\begin{equation}\label{p-1775-Corollary-2.4}
\alpha_1<\frac{e^{x+1/24(x+1/2)}\Gamma(x+1)}{(x+1/2)^{x+1/2}} \le\alpha_2
\end{equation}
holds for $x>0$ if and only if $\alpha_1\le\sqrt{\frac{2\pi}e}\,$ and $\alpha_2\ge\sqrt2\,e^{1/12}$.
It is clear that the left hand side inequality in~\eqref{p-1775-Corollary-2.4} is stronger than the corresponding one in~\eqref{p.236-Theorem-1-rew}, but, when $x\ge1$, the right hand side inequality in~\eqref{p-1775-Corollary-2.4} is weaker than the corresponding one in~\eqref{p.236-Theorem-1-rew}. The double inequality in~\cite[Theorem~2]{Lu-Ramanujan-2014} is weaker than~\eqref{p-1775-Corollary-2.4}.
In~\cite{Bukac-Sevli-Gamma.tex}, among other things, some necessary and sufficient conditions on $\lambda\ge0$ for the function
\begin{equation}\label{H-lambda(x)-dfn}
H_\lambda(x)=\frac{e^{x+1/24(x+\lambda)}\Gamma(x+1)}{(x+1/2)^{x+1/2}}
\end{equation}
to be logarithmically completely monotonic on $(0,\infty)$ were discovered.
\par
For more information on inequalities for bounding the gamma function $\Gamma$ and on the (logarithmically) complete monotonicity of functions involving $\Gamma$, please refer to the survey articles~\cite{bounds-two-gammas.tex, Wendel2Elezovic.tex-JIA, Wendel-Gautschi-type-ineq-Banach.tex} and plenty of references collected therein.
\par
When replacing $\theta(x)$ by $\frac{\vartheta(x)}{12x}$, Binet's formulas~\eqref{ebinet} and~\eqref{theta(x)-gamma} become
\begin{equation}\label{y}
\Gamma(x+1) =\sqrt{2\pi x}\Bigl(\frac{x}e\Bigr)
^{x}e^{\vartheta(x)/12x}, \quad x>0.
\end{equation}
We call $\vartheta(x)$ the variant remainder of Stirling's formula.
In~\cite{shi}, it was proved that the function $\vartheta(x)$ is strictly increasing on $[1,\infty )$. In~\cite{m1}, it was further proved that $\vartheta(x)$ is strictly decreasing on $(0,\beta)$ and strictly increasing on $(\beta,\infty)$, where $\beta =0.34142\dotsc$ is the unique positive real
number satisfying%
\begin{equation}
\ln \Gamma(\beta+1) +\beta\psi (\beta+1) -\ln\sqrt{2\pi}\,-2\beta\ln \beta+\beta=0.
\end{equation}
\par
For $x>-\frac12$, let
\begin{equation}\label{bs}
\Gamma(x+1) =\sqrt{2\pi}\,\biggl(\frac{x+1/2}{e}\biggr)^{x+1/2}e^{b(x)}
\end{equation}
and
\begin{equation}\label{w}
\Gamma(x+1) =\sqrt{2\pi}\,\biggl(\frac{x+1/2}{e}\biggr)^{x+1/2}e^{w(x)/12x}.
\end{equation}
We call $b(x)$ and $w(x)$ the remainder of Burnside's formula and the variant remainder of Burnside's formula respectively.
\begin{equation}
n!\sim\sqrt{2\pi}\,\biggl(\frac{n+1/2}{e}\biggr)^{n+1/2}, \quad n\to\infty
\end{equation}
established in~\cite{Burnside-1917}.
\par
In~\cite{qu}, it was proved that the functions $-b(x)$, $xb(x) +\frac{1}{24}$, and $w(x)+\frac12$ are completely monotonic on $\bigl(\frac12,\infty\bigr)$. It is clear that the complete monotonicity of $-b(x)$ on $\bigl(\frac12,\infty\bigr)$ may be derived from the logarithmically complete monotonicity of the function $\frac1{g_{\alpha}(x)}$ on $(0,\infty)$.
\par
The aim of this paper is to extend and generalize some results mentioned above. Our main results may be formulated as the following theorems.

\begin{thm}\label{Q(x)-CM-thm-1}
The remainder $b(x)$ and the variant remainder $w(x)$ of Burnside's formula have the following properties.
\begin{enumerate}
\item
For $x>-\frac12$, the remainders $b(x)$ and $w(x)$ of Burnside's formula have the integral representation
\begin{equation}\label{b(x)-int-eq}
b(x)=\frac{w(x)}{12x}
=\int_0^\infty\biggl(1-\frac{e^t}{2t}+\frac{1}{e^{2t}-1}\biggr) \frac1te^{-2(x+1)t}\td t
\end{equation}
and the function $-b(x)=-\frac{w(x)}{12x}$ is completely monotonic on $\bigl(-\frac12,\infty\bigr)$;
\item
For $x>0$, the function
\begin{equation}
xb(x)+\frac1{24}=\frac1{12}\biggl[w(x)+\frac1{2}\biggr]
=\frac12\int_0^\infty\frac{f_1(t)}{t^3(e^{2t}-1)^2}e^{-(2x+1)t}\td t
\end{equation}
and is completely monotonic on $(0,\infty)$, where
\begin{equation}\label{f-1(t)-dfn-eq}
f_1(t)=(t+2)e^{4 t}-2 t(2 t+1) e^{3 t}-2 (t+2)e^{2 t}+2t e^t +t+2
\end{equation}
is absolutely monotonic on $(0,\infty)$;
\item
For $x>0$, the function
\begin{equation}
\frac16-\frac{x}3-8x^2b(x)=\frac23\biggl[\frac14-\frac{x}2-xw(x)\biggr]
=\int_0^\infty\frac{f_2(t)}{t^4(e^{2t}-1)^3} e^{-(2x+1)t}\td t
\end{equation}
and is completely monotonic on $(0,\infty)$, where
\begin{equation}
\begin{split}
f_2(t)&=(t^2+4t+6)e^{6 t}-4 t(2t^2+2 t+1) e^{5 t}-3(t^2+4 t+6)e^{4 t}\\
&\quad-8t(t^2-t-1) e^{3 t}+3(t^2+4 t+6) e^{2 t}-4t e^t -t^2-4 t-6
\end{split}
\end{equation}
is absolutely monotonic on $(0,\infty)$;
\item
For $x>0$, the function
\begin{equation}
\begin{split}
16x^3b(x)+\frac23x^2-\frac13x+\frac{23}{180}
&=\frac43\biggl[x^2w(x)+\frac12x^2-\frac14x+\frac{23}{240}\biggr]\\
&=\int_0^\infty\frac{f_3(t)}{t^5(e^{2t}-1)^4} e^{-(2x+1)t}\td t
\end{split}
\end{equation}
and is completely monotonic on $(0,\infty)$, where
\begin{equation}
\begin{aligned}
f_3(t)&=(t^3+6 t^2+18 t+24)e^{8 t}-4t(4 t^3+6t^2+6 t+3) e^{7 t}\\
&\quad-4 (t^3+6 t^2+18 t+24)e^{6 t}-4 t (16 t^3-12 t-9)e^{5 t}\\
&\quad+6 (t^3+6 t^2+18 t+24) e^{4 t}-4t (4 t^3-6 t^2+6 t+9) e^{3 t}\\
&\quad-4 (t^3+6 t^2+18t+24)e^{2 t} +12 te^t +t^3+6 t^2+18 t+24
\end{aligned}
\end{equation}
is absolutely monotonic on $(0,\infty)$;
\item
For $x>-\frac12$, the function $(2x+1)b(x)+\frac1{12}$ has the integral representation
\begin{equation}
(2x+1)b(x)+\frac1{12}=\int_0^\infty\frac{h_1(t)}{\left(e^{2 t}-1\right)^2 t^3} e^{-(2x+1)t}\td t
\end{equation}
and is completely monotonic on $\bigl(-\frac12,\infty\bigr)$, where
\begin{equation}
h_1(t)=e^{4 t}- t(t+1)e^{3 t}-2 e^{2 t}- t(t-1)e^t +1
\end{equation}
is absolutely monotonic on $(0,\infty)$;
\item
For $x>-\frac12$, the function $-(x+1)b(x)-\frac1{24}$ has the integral representation
\begin{equation}
-(x+1)b(x)-\frac1{24}=\frac14\int_0^\infty\frac{h_2(t)}{t^3(e^{2t}-1)^2}e^{-(2x+1)t}\td t
\end{equation}
and is completely monotonic on $\bigl(-\frac12,\infty\bigr)$, where
\begin{equation}
h_2(t)=e^{4 t} (t-2)+2 e^{3 t} t-2 e^{2 t} (t-2)+2 e^t t (2 t-1)+t-2
\end{equation}
is absolutely monotonic on $(0,\infty)$;
\item
For $x>-\frac12$, the function $-(x+1)^2b(x)-\frac1{24}x-\frac1{16}$ has the integral representation
\begin{equation}
-(x+1)^2b(x)-\frac1{24}x-\frac1{16} =\frac18\int_0^\infty \frac{h_3(t)}{(e^{2 t}-1)^3 t^4} e^{-(2x+1)t}\td t
\end{equation}
and is completely monotonic on $\bigl(-\frac12,\infty\bigr)$, where
\begin{equation}
\begin{split}
h_3(t)&=e^{6 t} (t^2-4 t+6)-4 e^{5 t} t-3 e^{4 t} (t^2-4 t+6)-8 e^{3 t} (t^2+t-1) t\\
&\quad+3 e^{2 t} (t^2-4 t+6)-4 e^t (2 t^2-2 t+1) t-t^2+4 t-6
\end{split}
\end{equation}
is absolutely monotonic on $(0,\infty)$;
\item
For $x>-\frac12$, the function $-(x+1)^3b(x)-\frac{x^2}{24}-\frac{5 x}{48}-\frac{203}{2880}$ has the integral representation
\begin{equation}
-(x+1)^3b(x)-\frac{x^2}{24}-\frac{5 x}{48}-\frac{203}{2880}
=\int_0^\infty \frac{h_4(t)}{(e^{2 t}-1)^4 t^5} e^{-(2x+1)t}\td t
\end{equation}
and is completely monotonic on $\bigl(-\frac12,\infty\bigr)$, where
\begin{equation}
\begin{aligned}
h_4(t)&=e^{8 t} (t^3-6 t^2+18 t-24)+12 e^{7 t} t-4 e^{6 t} (t^3-6 t^2+18 t-24)\\
&\quad+4 e^{5 t} (4 t^3+6 t^2+6 t-9) t+6 e^{4 t} (t^3-6 t^2+18 t-24)\\
&\quad+4 e^{3 t} (16 t^3-12 t+9) t-4 e^{2 t} (t^3-6 t^2+18 t-24)\\
&\quad+4 e^t (4 t^3-6 t^2+6 t-3) t+t^3-6 t^2+18 t-24
\end{aligned}
\end{equation}
is absolutely monotonic on $(0,\infty)$.
\end{enumerate}
\end{thm}

\begin{thm}\label{Q(x)-LCM-thm-2}
The functions
\begin{equation*}
\biggl[\frac{\Gamma(x+1)}{\sqrt{2\pi}\,} \biggl(\frac{e}{x+1/2}\biggr)^{x+1/2}\biggr]^x,\quad
\biggl[\frac{\sqrt{2\pi}\,}{\Gamma(x+1)} \biggl(\frac{x+1/2}{e}\biggr)^{x+1/2}\biggr]^{8x^2} e^{-x/3},
\end{equation*}
and
\begin{equation*}
\biggl[\frac{\Gamma(x+1)}{\sqrt{2\pi}\,} \biggl(\frac{e}{x+1/2}\biggr)^{x+1/2}\biggr]^{16x^3} e^{x(2x-1)/3}
\end{equation*}
are logarithmically completely monotonic on $(0,\infty)$.
\par
The functions
\begin{gather*}
\frac{\sqrt{2\pi}\,}{\Gamma(x+1)} \biggl(\frac{x+1/2}{e}\biggr)^{x+1/2},\quad
\biggl[\frac{\Gamma(x+1)}{\sqrt{2\pi}\,} \biggl(\frac{e}{x+1/2}\biggr)^{x+1/2}\biggr]^{2x+1},\\
\biggl[\frac{\sqrt{2\pi}\,}{\Gamma(x+1)} \biggl(\frac{x+1/2}{e}\biggr)^{x+1/2}\biggr]^{x+1},\quad
\biggl[\frac{\sqrt{2\pi}\,}{\Gamma(x+1)} \biggl(\frac{x+1/2}{e}\biggr)^{x+1/2}\biggr]^{(x+1)^2} e^{-x/24},
\end{gather*}
and
\begin{equation*}
\biggl[\frac{\sqrt{2\pi}\,}{\Gamma(x+1)} \biggl(\frac{x+1/2}{e}\biggr)^{x+1/2}\biggr]^{(x+1)^3} e^{-(2x+5)/48}
\end{equation*}
are logarithmically completely monotonic on $\bigl(-\frac12,\infty\bigr)$.
\end{thm}

\section{Proofs of main results}

Now we start out to prove Theorems~\ref{Q(x)-CM-thm-1} and~\ref{Q(x)-LCM-thm-2}.

\begin{proof}[Proof of Theorem~\ref{Q(x)-CM-thm-1}]
From~\eqref{bs} and~\eqref{w}, it follows that
\begin{equation}\label{ss}
w(x) =12xb(x)=12x\biggl[\ln\Gamma(x+1)-\frac12\ln(2\pi) -\biggl(x+\frac12\biggr)\ln\biggl(x+\frac1{2}\biggr)+x +\frac12\biggr].
\end{equation}
By Binet's formula~\eqref{ebinet}, we have
\begin{equation*}
\ln \Gamma(x+1)= \biggl(x+\frac12\biggr)\ln(x+1)-x-1+\ln \sqrt{2\pi}\,+\theta(x+1).
\end{equation*}
Substituting this into~\eqref{ss} results in
\begin{equation}\label{b(x)-simp-eq}
b(x)=\frac{w(x)}{12x} =\frac12\biggl[(2x+1)\ln\biggl(1+\frac1{2x+1}\biggr) -1\biggr]+\theta(x+1).
\end{equation}
It was listed in~\cite[pp.~322\nobreakdash--323, Entry~46]{Schilling-Song-Vondracek-2nd} that
\begin{equation}\label{Entry46-Bernstein}
x^2 \ln\biggl(1+\frac{1}{x}\biggr) -x+\frac{1}{2}
=\int_0^\infty \frac{2-(t^2+2 t+2) e^{-t}}{t^3}e^{-xt}\td t, \quad x>0.
\end{equation}
Making use of~\eqref{Entry46-Bernstein} in~\eqref{b(x)-simp-eq} produces
\begin{align*}
b(x) &=\frac1{2(2x+1)}\biggl[\int_0^\infty\frac{2-(t^2+2t+2) e^{-t}}{t^3}e^{-(2x+1)t}\td t -\frac12\biggr]+\theta(x+1)\\
&=\frac1{2(2x+1)}\biggl\{\int_0^\infty\frac{\td}{\td t}\biggl[\frac{e^{-t} (1+t-e^t)}{t^2}\biggr] e^{-(2x+1)t}\td t -\frac12\biggr\} +\theta(x+1)\\
&=\frac1{2(2x+1)}\biggl[\frac{e^{-t} (1+t-e^t)}{t^2}e^{-(2x+1)t}\bigg|_{t=0}^{t=\infty} \\
&\quad +(2x+1)\int_0^\infty\frac{e^{-t} (1+t-e^t)}{t^2} e^{-(2x+1)t}\td t -\frac12\biggr] +\theta(x+1)\\
&=\frac12\int_0^\infty\frac{1+t-e^t}{t^2} e^{-2(x+1)t}\td t+\theta(x+1)\\
&=\frac12\int_0^\infty\frac{1+t-e^t}{t^2} e^{-2(x+1)t}\td t +\int_{0}^{\infty}\biggl(\frac{1}{e^{t}-1}-\frac{1}{t}
+\frac12\biggr)\frac{e^{-(x+1)t}}{t}\td t\\
&=\frac12\int_0^\infty\frac{1+t-e^t}{t^2} e^{-2(x+1)t}\td t +\int_{0}^{\infty}\biggl(\frac{1}{e^{2t}-1}-\frac{1}{2t}
+\frac12\biggr)\frac{e^{-2(x+1)t}}{t}\td t\\
&=\int_0^\infty\biggl[\frac{1+t-e^t}{2t}+\biggl(\frac{1}{e^{2t}-1}-\frac{1}{2t}
+\frac12\biggr)\biggr]\frac{e^{-2(x+1)t}}t\td t\\
&=\int_0^\infty\biggl(1-\frac{e^t}{2t}+\frac{1}{e^{2t}-1}\biggr)\frac{e^{-t}}te^{-(2x+1)t}\td t.
\end{align*}
\par
Since the function
\begin{equation*}
1-\frac{e^t}{2t}+\frac{1}{e^{2t}-1}=-\frac{e^t (e^{2 t}-2t e^t-1)}{2t (e^{2t}-1)}<0
\end{equation*}
on $(0,\infty)$, from the integral representation~\eqref{b(x)-int-eq} of $b(x)$, it may be easily deduced that the function $-b(x)$ is completely monotonic on $\bigl(-\frac12,\infty\bigr)$.
\par
Utilizing the integral representation~\eqref{b(x)-int-eq} of $b(x)$ and integrating by parts reveal
\begin{align*}
-xb(x)&=\frac12\int_0^\infty\biggl(1-\frac{e^t}{2t}+\frac{1}{e^{2t}-1}\biggr) \frac{e^{-2t}}t\frac{\td e^{-2xt}}{\td t}\td t\\
&=\frac12\biggl[\frac1{12}-\int_0^\infty\frac{e^{-t}f_1(t)}{2 t^3(e^{2t}-1)^2}e^{-2xt}\td t\biggr]\\
&=\frac12\biggl[\frac1{12}+\frac1{4x}\int_0^\infty\frac{e^{-t}f_1(t)}{t^3(e^{2t}-1)^2}\frac{\td e^{-2xt}}{\td t}\td t\biggr]\\
&=\frac12\biggl\{\frac1{12}+\frac1{4x}\biggl[-\frac16+\int_0^\infty\frac{e^{-t} f_2(t)}{t^4(e^{2t}-1)^3} e^{-2xt}\td t\biggr]\biggr\}\\
&=\frac12\biggl\{\frac1{12}+\frac1{4x}\biggl[-\frac16-\frac1{2x} \int_0^\infty\frac{e^{-t}f_2(t)}{t^4(e^{2t}-1)^3} \frac{\td e^{-2xt}}{\td t}\td t\biggr]\biggr\}\\
&=\frac12\biggl\{\frac1{12}+\frac1{4x}\biggl[-\frac16-\frac1{2x} \biggl(-\frac{23}{180}+\int_0^\infty\frac{e^{-t}f_3(t)}{t^5(e^{2t}-1)^4} e^{-2xt}\td t\biggr)\biggr]\biggr\}.
\end{align*}
\par
By straightforward computation, we have
\begin{align*}
f_1'(t)&=e^{4 t} (4 t+9)-2 e^{3 t} (6 t^2+7 t+1)-2 e^{2 t} (2 t+5)+2 e^t (t+1)+1\\
&\to0\quad\text{as $t\to0$},\\
f_1''(t)&=2 e^t [4 e^{3 t} (2 t+5)-e^{2 t} (18 t^2+33 t+10)-4 e^t (t+3)+t+2]\\
&\triangleq2 e^tf_{11}(t)\\
&\to0\quad\text{as $t\to0$},\\
f_{11}'(t)&=4 e^{3 t} (6 t+17)-e^{2 t} (36 t^2+102 t+53)-4 e^t (t+4)+1\\
&\to0\quad\text{as $t\to0$},\\
f_{11}''(t)&=4 e^t [3 e^{2 t} (6 t+19)-e^t (18 t^2+69 t+52)-t-5]\\
&\triangleq4 e^t f_{12}(t)\\
&\to0\quad\text{as $t\to0$},\\
f_{12}'(t)&=12 e^{2 t} (3 t+11)-e^t (18 t^2+105 t+121)-1\\
&\to10\quad\text{as $t\to0$},\\
f_{12}''(t)&=e^t [12 e^t (6 t+25)-18 t^2-141 t-226]\\
&\triangleq e^tf_{13}(t)\\
&\to74\quad\text{as $t\to0$},\\
f_{13}'(t)&=3 [4 e^t (6 t+31)-12 t-47]\\
&\to231\quad\text{as $t\to0$},\\
f_{13}''(t)&=12[e^t (6 t+37)-3]\\
&\to408\quad\text{as $t\to0$},\\
f_{13}'''(t)&=12 e^t (6 t+43).
\end{align*}
As a result, since the product of finitely many absolutely monotonic functions is still absolutely monotonic, the function $f_1(t)$ is absolutely monotonic on $(0,\infty)$. This means that the function $xb(x)+\frac1{24}$ is completely monotonic on $(0,\infty)$.
\par
By direct calculation, we have
\begin{align*}
f_2'(t)&=e^{6 t} (6 t^2+26 t+40)-4 e^{5 t} (10 t^3+16 t^2+9 t+1)-6 e^{4 t} (2 t^2+9 t+14)\\
&\quad-e^{3 t} (24 t^3-40 t-8)+6 e^{2 t} (t^2+5 t+8)-4 e^t (t+1)-2 (t+2)\\
&\to0\quad \text{as $t\to0$},\\
f_2''(t)&=2 \bigl[e^{6 t} (18 t^2+84 t+133)-2 e^{5 t} (50 t^3+110 t^2+77 t+14)\\
&\quad-3 e^{4 t} (8 t^2+40 t+65)-4 e^{3 t} (9 t^3+9 t^2-15 t-8)\\
&\quad+e^{2 t} (6 t^2+36 t+63)-2 e^t (t+2)-1\bigr]\\
&\to0\quad \text{as $t\to0$},\\
f_2'''(t)&=4 e^t \bigl[9 e^{5 t} (6 t^2+30 t+49)-e^{4 t} (250 t^3+700 t^2+605 t+147)\\
&\quad-6 e^{3 t} (8 t^2+44 t+75)-6 e^{2 t} (9 t^3+18 t^2-9 t-13)\\
&\quad+e^t (6 t^2+42 t+81)-t-3\bigr]\\
&\triangleq4 e^tf_{21}(t)\\
&\to0\quad \text{as $t\to0$},\\
f_{21}'(t)&=9 e^{5 t} (30 t^2+162 t+275)-e^{4 t} (1000 t^3+3550 t^2+3820 t+1193)\\
&\quad-6 e^{3 t} (24 t^2+148 t+269)-6 e^{2 t} (18 t^3+63 t^2+18 t-35)\\
&\quad+3 e^t (2 t^2+18 t+41)-1\\
&\to0\quad \text{as $t\to0$},\\
f_{21}''(t)&=e^t \bigl[9 e^{4 t} (150 t^2+870 t+1537)-4 e^{3 t} (1000 t^3+4300 t^2+5595 t+2148)\\
&\quad-6 e^{2 t} (72 t^2+492 t+955)-12 e^t (18 t^3+90 t^2+81 t-26)\\
&\quad+6 t^2+66 t+177\bigr]\\
&\triangleq e^tf_{22}(t)\\
&\to0\quad \text{as $t\to0$},\\
f_{22}'(t)&=2 \bigl[9 e^{4 t} (300 t^2+1890 t+3509)-2 e^{3 t}(3000 t^3+15900 t^2+25385 t\\
&\quad+12039)-6 e^{2 t} (72 t^2+564 t+1201)-6 e^t (18 t^3+144 t^2+261 t+55)\\
&\quad+6 t+33\bigr]\\
&\to0\quad \text{as $t\to0$},\\
f_{22}''(t)&=4 \bigl[9 e^{4 t} (600 t^2+4080 t+7963)-e^{3 t}(9000 t^3+56700 t^2\\
&\quad+107955 t+61502)-6 e^{2 t} (72 t^2+636 t+1483)\\
&\quad-3 e^t (18 t^3+198 t^2+549 t+316)+3\bigr]\\
&\to1288\quad \text{as $t\to0$},\\
f_{22}'''(t)&=12 e^t \bigl[12 e^{3 t}(600 t^2+4380 t+8983)-e^{2 t}(9000 t^3+65700 t^2+145755 t\\
&\quad+97487)-4 e^t (72 t^2+708 t+1801)-18 t^3-252 t^2-945 t-865\bigr]\\
&\triangleq12 e^tf_{23}(t)\\
&\to26880\quad \text{as $t\to0$},\\
f_{23}'(t)&=36 e^{3 t} (600 t^2+4780 t+10443)-e^{2 t}(18000 t^3+158400 t^2+422910 t\\
&\quad+340729)-4 e^t (72 t^2+852 t+2509)-9 (6 t^2+56 t+105)\\
&\to24238\quad \text{as $t\to0$},\\
f_{23}''(t)&=4 \bigl[9 e^{3 t} (1800 t^2+15540t+36109)-e^{2t}(9000 t^3+92700 t^2+290655 t\\
&\quad+276092)-e^t (72 t^2+996 t+3361)-9 (3 t+14)\bigr]\\
&\to181608\quad \text{as $t\to0$},\\
f_{23}'''(t)&=4 \bigl[81 e^{3 t} (600 t^2+5580 t+13763)-e^{2 t} (18000 t^3+212400 t^2+766710 t\\
&\quad+842839)-e^t (72 t^2+1140 t+4357)-27\bigr]\\
&\to1070320\quad \text{as $t\to0$},\\
f_{23}^{(4)}(t)&=4 e^t \bigl[243 e^{2 t} (600 t^2+5980 t+15623)-4 e^t (9000 t^3+119700 t^2+489555 t\\
&\quad+613097)-72 t^2-1284 t-5497\bigr]\\
&\triangleq4 e^tf_{24}(t)\\
&\to5354016\quad \text{as $t\to0$},\\
f_{24}'(t)&=2 \bigl[243 e^{2 t} (600 t^2+6580 t+18613)-2e^t(9000 t^3+146700 t^2+728955 t\\
&\quad+1102652)-72 t-642\bigr]\\
&\to4634026\quad \text{as $t\to0$},\\
f_{24}''(t)&=4 \bigl[243 e^{2 t} (600 t^2+7180 t+21903)-e^t (9000 t^3+173700 t^2\\
&\quad+1022355 t+1831607)-36\bigr]\\
&\to13963144\quad \text{as $t\to0$},\\
f_{24}'''(t)&=4 e^t \bigl[486 e^t (600 t^2+7780 t+25493)-9000 t^3-200700 t^2\\
&\quad-1369755 t-2853962\bigr]\\
&\triangleq4 e^tf_{25}(t)\\
&\to38142544\quad \text{as $t\to0$},\\
f_{25}'(t)&=9 [54 e^t (600 t^2+8980 t+33273)-5 (600 t^2+8920 t+30439)]\\
&\to14800923\quad \text{as $t\to0$},\\
f_{25}''(t)&=18[27 e^t(600 t^2+10180 t+42253)-100 (30 t+223)]\\
&\to20133558\quad \text{as $t\to0$},\\
f_{25}'''(t)&=54[9 e^t(600 t^2+11380 t+52433)-1000]\\
&\to25428438\quad \text{as $t\to0$},\\
f_{25}^{(4)}(t)&=486 e^t(600 t^2+12580 t+63813).
\end{align*}
Therefore, by the fact that the product of finitely many absolutely monotonic functions is still absolutely monotonic, we see that the function $f_2(t)$ is absolutely monotonic on $(0,\infty)$. This implies that the function $\frac16-\frac{x}3-8x^2b(x)$ is completely monotonic on $(0,\infty)$.
\par
By similar arguments to proofs of the absolute monotonicity of $f_1$ and $f_2$, we may verify that the function $f_3$ is also absolutely monotonic on $(0,\infty)$. Consequently, the function $16x^3b(x)+\frac23x^2-\frac13x+\frac{23}{180}$ is completely monotonic on $(0,\infty)$.
\par
Employing the integral representation~\eqref{b(x)-int-eq} and integrating by parts gives
\begin{align*}
-(2x+1)b(x)&=\int_0^\infty\biggl(1-\frac{e^t}{2t}+\frac{1}{e^{2t}-1}\biggr) \frac{e^{-t}}t\frac{\td e^{-(2x+1)t}}{\td t}\td t\\
&=\frac1{12}-\int_0^\infty\frac{h_1(t)}{(e^{2 t}-1)^2 t^3} e^{-(2x+1)t}\td t.
\end{align*}
The verification of the absolute monotonicity of $h_1(t)$ is same as that of $f_1$ and $f_2$. Accordingly, the function $(2x+1)b(x)+\frac1{12}$ is completely monotonic on $\bigl(-\frac12,\infty\bigr)$.
\par
Using the integral representation~\eqref{b(x)-int-eq} and integrating by parts acquires
\begin{gather*}
-2(x+1)b(x)=\int_0^\infty\biggl(1-\frac{e^t}{2t}+\frac{1}{e^{2t}-1}\biggr) \frac1t\frac{\td e^{-2(x+1)t}}{\td t}\td t\\
=\frac1{12}+\frac12\int_0^\infty\frac{e^th_2(t)}{t^3(e^{2t}-1)^2}e^{-2(x+1)t}\td t\\
=\frac1{12}-\frac1{4(x+1)} \int_0^\infty \frac{e^th_2(t)}{t^3(e^{2t}-1)^2} \frac{\td e^{-2(x+1)t}}{\td t}\td t\\
=\frac1{12}-\frac1{4(x+1)} \biggl[-\frac1{6}-\int_0^\infty \frac{e^th_3(t)}{(e^{2 t}-1)^3 t^4} e^{-2(x+1)t}\td t\biggr]\\
=\frac1{12}-\frac1{4(x+1)} \biggl[-\frac1{6}+\frac1{2(x+1)}\int_0^\infty \frac{e^th_3(t)}{(e^{2 t}-1)^3 t^4} \frac{\td e^{-2(x+1)t}}{\td t}\td t\biggr]\\
=\frac1{12}-\frac1{4(x+1)} \biggl\{-\frac1{6}+\frac1{2(x+1)}\biggl[-\frac{23}{180}-\int_0^\infty \frac{e^th_4(t)}{(e^{2 t}-1)^4 t^5} e^{-2(x+1)t}\td t\biggr]\biggr\}.
\end{gather*}
By similar arguments to proofs of the absolute monotonicity of $f_1$ and $f_2$, we may verify that the functions $h_2$, $h_3$, and $h_4$ are absolutely monotonic on $(0,\infty)$. Consequently, the functions $-(x+1)b(x)-\frac1{24}$, $-(x+1)^2b(x)-\frac1{24}x-\frac1{16}$, and $-(x+1)^3b(x)-\frac{x^2}{24}-\frac{5 x}{48}-\frac{203}{2880}$ are completely monotonic on $\bigl(-\frac12,\infty\bigr)$. The proof of Theorem~\ref{Q(x)-CM-thm-1} is complete.
\end{proof}

\begin{proof}[Proof of Theorem~\ref{Q(x)-LCM-thm-2}]
This follows from reformulation of the functions involving the remainder $b(x)$ in Theorem~\ref{Q(x)-CM-thm-1}.
\end{proof}

\section{Remarks}

For better understanding our main results, we list several remarks as follows.

\begin{rem}
In~\cite{mpf-1993, Schilling-Song-Vondracek-2nd, widder}, we may find the classical Bernstein-Widder theorem which reads that a function $f$ is completely monotonic on $(0,\infty)$ if and only if it is a Laplace transform of some nonnegative measure $\mu$, that is,
\begin{equation} \label{berstein-1}
f(x)=\int_0^\infty e^{-xt}\td\mu(t),
\end{equation}
where $\mu(t)$ is non-decreasing and the integral converges for $0<x<\infty$.
In~\cite{CBerg, absolute-mon-simp.tex, compmon2, minus-one, Schilling-Song-Vondracek-2nd}, we may search out that any logarithmically completely monotonic function must be a completely monotonic function, but not conversely. To some extent, these reveal the significance and meanings of our main results.
\end{rem}

\begin{rem}
By the monotonicity in Theorem~\ref{Q(x)-LCM-thm-2}, we may derive some double inequalities for bounding the function
\begin{equation}
\frac{\Gamma(x+1)}{\sqrt{2\pi}\,} \biggl(\frac{e}{x+1/2}\biggr)^{x+1/2}
\end{equation}
or its reciprocal on the intervals $(0,\infty)$ and $\bigl(-\frac12,\infty\bigr)$. These inequalities would be better than~\eqref{p.236-Theorem-1-rew} and~\eqref{p-1775-Corollary-2.4}.
\end{rem}

\begin{rem}
From~\eqref{y} and~\eqref{w}, it follows that
\begin{equation*}
\sqrt{x}\,\Bigl(\frac{x}e\Bigr)^{x}e^{\vartheta(x)/12x}
=\biggl(\frac{x+1/2}{e}\biggr)^{x+1/2}e^{w(x) /12x}
\end{equation*}
which may be rewritten as
\begin{equation}\label{s}
\vartheta(x) =w(x)-6x+12x\biggl(x+\frac12\biggr)\ln\biggl(1+\frac1{2x}\biggr)
\end{equation}
and
\begin{equation}\label{s-rev}
\theta(x) =b(x)+\biggl(x+\frac12\biggr)\ln\biggl(1+\frac1{2x}\biggr)-\frac12.
\end{equation}
Combining these two identities with Theorem~\ref{Q(x)-CM-thm-1}, we may deduce the complete monotonicity of some functions related to the remainder $\theta(x)$ and the variant remainder $\vartheta(x)$ of Binet's formula. Furthermore, we may deduce some double inequalities for bounding the function
\begin{equation}
\frac{e^x\Gamma(x+1)}{x^x\sqrt{2\pi x}\,}
\end{equation}
or its reciprocal on the interval $(0,\infty)$.
\end{rem}

\begin{rem}
In~\cite[Theorem~1.2]{Lu-JNT-2014}, it was given that the inequality
\begin{equation}\label{Lu-JNT-2014-th2-ineq}
\Gamma(x+1)<\sqrt{2\pi}\,\biggl(\frac{x+1/2}e\biggr)^{x+1/2} \biggr[1-\frac{k}{24x}+\biggl(\frac{k^2}{1152}+\frac{k}{48}\biggr)\frac1{x^2}\biggr]^{1/k}
\end{equation}
is valid for $x\ge m_1$ and for any positive integer $k$, where $m_1\ge0$ is a constant depending on $k$. To compare the right hand side inequality in~\eqref{p-1775-Corollary-2.4} with~~\eqref{Lu-JNT-2014-th2-ineq}, it suffices to discuss the inequality
\begin{equation}
\frac{e^{7/12}}{\sqrt{\pi}\,}\le e^{1/24(x+1/2)} \biggr[1-\frac{k}{24x} +\biggl(\frac{k^2}{1152}+\frac{k}{48}\biggr)\frac1{x^2}\biggr]^{1/k}.
\end{equation}
If letting $x\to\infty$, the above inequality becomes $e^{7/12}=1.79\dotsc\le\sqrt{\pi}\,=1.77\dotsc$. This contradiction implies that, when $x$ is sufficiently large, the inequality~\eqref{Lu-JNT-2014-th2-ineq} is better than the right hand side inequality in~\eqref{p-1775-Corollary-2.4}.
\par
In~\cite[Theorem~1.2]{Lu-Wang-JMAA-2013}, it was also deduced that the inequality
\begin{equation}\label{Lu-Wang-JMAA-2013-th1.2-ineq}
\Gamma(x+1)<\sqrt{2\pi x}\,\Bigl(\frac{x}e\Bigr)^x \biggl(1+\frac{k}{12x}+\frac{k^2}{288x^2}\biggr)^{1/k}
\end{equation}
holds for $x\ge m_1$ and $1\le k\le5$ and reverses for $x\ge m_2$ and $k\ge6$, where $m_1$ and $m_2$ are constants depending on $k$.
\par
Because the inequalities~\eqref{Lu-JNT-2014-th2-ineq} and~\eqref{Lu-Wang-JMAA-2013-th1.2-ineq} are derived from an asymptotic approximation of the gamma function $\Gamma(x+1)$, they may be more accurate when $x$ is sufficiently large, but not, even invalid, when $x$ is close to zero.
\end{rem}

\begin{rem}
Now we consider the functions
\begin{equation}
F(x) =1+4x-8x\biggl(x+\frac12\biggr) \ln \biggl(1+\frac1{2x}\biggr)
\end{equation}
and
\begin{equation}
G(x)=\biggl(x+\frac12\biggr)\ln\biggl(1+\frac1{2x}\biggr)-\frac12
\end{equation}
appeared in~\eqref{s} and~\eqref{s-rev}. It is obvious that $F(x)=1-8xG(x)$.
\par
We claim that the functions $F(x)$ and $G(x)$ are completely monotonic on $(0,\infty)$. This may be verified as follows.
A direct computation gives
\begin{align*}
F'(x)&=4\biggl[2-(4 x+1) \ln \biggl(1+\frac{1}{2 x}\biggr)\biggr],\\
F''(x)&=\frac{4(4 x+1)}{x (2 x+1)}-16\ln \biggl(1+\frac{1}{2 x}\biggr),\\
F'''(x)&=-\frac4{x^2 (2 x+1)^2}.
\end{align*}
It is clear that $\lim_{x\to\infty}F''(x)=0$. It is not difficult to see that
\begin{equation*}
F'(x)=4\biggl[2-2\ln \biggl(1+\frac{1}{2 x}\biggr)^{2x}-\ln \biggl(1+\frac{1}{2 x}\biggr)\biggr]
\to4(2-2\ln e-0)=0
\end{equation*}
and
\begin{align*}
F(x) &=1-2\ln \biggl(1+\frac1{2x}\biggr)^{2x}+8x^2\biggl[\frac1{2x}-\ln \biggl(1+\frac1{2x}\biggr)\biggr]\\
&\to1-2\ln e+2\lim_{u\to0^+}\frac{u-\ln(1+u)}{u^2}\\
&=0
\end{align*}
as $x\to\infty$, where $u=\frac1{2x}$. Hence, by virtue of $F'''(x)>0$ on $(0,\infty)$, we may conclude $F(x)>0$, $F'(x)<0$, and $F''(x)>0$. Furthermore, by the facts that the functions $\frac1x$ and $\frac1{1+2x}$ are completely monotonic on $(0,\infty)$ and that the product of finitely many completely monotonic functions is also a completely monotonic function, we may see that the function $-F'''(x)$ is completely monotonic on $(0,\infty)$. In a word, the function $F(x)$ is completely monotonic on $(0,\infty)$.
\par
The complete monotonicity of $G(x)$ may also be verified straightforwardly.
\end{rem}

\subsection*{Acknowledgements}
The author thanks the anonymous referee for his/her careful corrections to and valuable comments on the original version of this paper.
\par
This work was partially supported by the NNSF under Grant No.~11361038 of China and the Foundation of the Research Program of Science and Technology at Universities of Inner Mongolia Autonomous Region under Grant number No.~NJZY14191 and No.~NJZY14192, China.

\end{document}